\documentclass[final]{amsart}
\usepackage{graphicx}
\usepackage{microtype}

\usepackage[numbers]{natbib}
\newcommand{\ym}[1]{}

\usepackage{amsmath}
\usepackage{amssymb}
\usepackage[all,cmtip]{xy}

\usepackage{enumitem}
\usepackage{hyperref}
\usepackage[capitalise,noabbrev]{cleveref}

\usepackage{ifdraft}
\ifdraft{\input{draft.tex}}{}

\theoremstyle{plain}
\newtheorem*{theorem*}{Main Theorem}
\newtheorem{theorem}{Theorem}[section]
\newtheorem{proposition}[theorem]{Proposition}
\newtheorem{lemma}[theorem]{Lemma}
\newtheorem{corollary}[theorem]{Corollary}
\theoremstyle{definition}
\newtheorem{definition}[theorem]{Definition}
\newtheorem{example}[theorem]{Example}

\theoremstyle{remark}
\newtheorem{remark}[theorem]{Remark}

\newcommand{\ZZ}{\mathbf{Z}}

\newcommand{\cofib}{\operatorname{cofib}}

\newcommand{\QCoh}{\operatorname{D}_{\textnormal{qcoh}}}
\newcommand{\Coh}{\operatorname{D}_{\textnormal{coh}}^{\textnormal{b}}}
\newcommand{\Perf}{\operatorname{D}^{\textnormal{perf}}}

\newcommand{\CAlg}{\operatorname{CAlg}}

\newcommand{\Spec}{\operatorname{Spec}}

\newcommand{\unit}{\mathbf{1}}
\newcommand{\shv}[1]{\mathcal{#1}}
\newcommand{\cat}[1]{\mathcal{#1}}
\newcommand{\Cat}[1]{\mathsf{#1}}

\newcommand{\llangle}{\langle\!\langle}
\newcommand{\rrangle}{\rangle\!\rangle}

\let\autocite\cite

\title{Quasiexcellence implies strong generation}
\author{Ko Aoki}
\address{
    Department of Mathematics,
    Tokyo Institute of Technology, 2--12--1 \=Ookayama, Meguro-ku,
    Tokyo 152--8551, Japan
}
\email{aoki.k.an@m.titech.ac.jp}
\date{\today}

\begin{document}

\begin{abstract}
    We prove that
    the bounded derived category
    of coherent sheaves on a quasicompact separated
    quasiexcellent scheme of finite dimension
    has a strong generator in the sense
    of Bondal--Van~den~Bergh.
    This extends a recent result of Neeman
    and is new even in the affine case.
    The main ingredient includes
    Gabber's weak local uniformization theorem
    and the notions of boundedness and descendability
    of a morphism of schemes.
\end{abstract}

\maketitle

\section{Introduction}\label{s_intro}

In~\autocite{BondalVandenBergh03},
Bondal and Van~den~Bergh introduced
the notion of strong generator of a triangulated category.
It is useful because
under the existence of a strong generator and the properness assumption,
a certain appealing form of the Brown representability theorem holds;
see \autocite[Theorem~1.3]{BondalVandenBergh03} for the precise statement.

We wish to know
that many of the naturally arising triangulated categories
have strong generators.
First, they proved in
\autocite[Theorem~3.1.4]{BondalVandenBergh03}
that if $ X $
is a quasicompact separated scheme smooth over a field,
the bounded derived category of coherent sheaves
$ \Coh(X) $, which is equal to $ \Perf(X) $ in this case,
admits a strong generator.
Recently, in~\autocite{Neeman},
Neeman generalized this result to the case
where $ X $ is a separated noetherian scheme
essentially of finite type over
an excellent scheme of dimension $ \leq 2 $.

The main result of this paper is the following
further generalization of Neeman's result,
which we demonstrate in \cref{s_main}.

\begin{theorem*}\label{44910ebe3e}
    If $ X $ is
    a quasicompact separated quasiexcellent scheme
    of finite dimension,
    $ \Coh(X) $ has a strong generator.
\end{theorem*}

Neeman used
de~Jong's theorem on alterations to prove his result.
Our strategy is to rather use weak local uniformizations,
whose existence for a quasiexcellent scheme
is already known due to Gabber.
In order to do so, we should contemplate on
how h~covers (or alteration covers) and strong generators interact with each other.
We found the two notions of boundedness and descendability
of a morphism,
which we treat in \cref{s_bounded} and \cref{s_descendable}
respectively, useful when considering that problem.

\subsection*{Convention}

To simplify the exposition,
we always regard $ \QCoh $ and $ \Perf $
of a quasicompact quasiseparated scheme
and $ \Coh $ of a noetherian scheme
as stable $ \infty $-categories.
Moreover, all pullback, pushforward, and tensor product
functors in this paper mean the derived ones,
so we put neither $ \mathrm{L} $ nor~$ \mathrm{R} $
to indicate how they are derived.

\subsection*{Acknowledgments}

I would like to thank Shane Kelly for introducing~\autocite{Neeman} to me,
and Amnon Neeman and Michael Temkin for
answering several questions.
I also thank Amnon for providing helpful feedback on a draft.

\section{Basic definitions}\label{s_def}

We first review some basic notions.
Our notation and terminology may slightly differ
from the common ones.

\begin{definition}\label{6c6b549161}
    Let $ \cat{C} $ be a stable $ \infty $-category.
    We call a collection of objects of~$ \cat{C} $
    \emph{closed} if
    it is closed under
    finite coproducts and direct summands.
    For a collection $ S \subset \cat{C} $,
    let $ \langle S \rangle $ denote
    the smallest closed subcollection containing~$ S $.
    For two closed subcollections~$ S $,~$ T $,
    we let $ S \star T $ denote the smallest closed subcollection
    containing an object $ C $ such that
    there exists $ C' \in S $, $ C'' \in T $, and
    a cofiber sequence $ C' \to C \to C'' $.
\end{definition}

We often omit curly braces to simplify the notation;
for example, $ \langle C \rangle $ for an object~$ C $
means what should be denoted by $ \langle \{C\} \rangle $, to be exact.

\begin{remark}\label{9d532cb5b6}
    The operation~$ \star $ was considered for example
    in \autocite[Section~1.3]{BBD}, where they proved its associativity.
    (Note that $ \star $ differs from what was denoted by~$ * $ there
    in that we apply the closure operation.)
    Therefore, following the usual pattern,
    we write $ S^{\star n} $
    for the ``$ n $th power'' of a closed subcollection~$ S $ for $ n > 0 $
    and $ S^{\star 0} $ for the collection consisting of
    zero objects.
\end{remark}

\begin{definition}\label{91f78d7a51}
    An object $ C \in \cat{C} $ of
    a stable $ \infty $-category~$ \cat{C} $
    is called
    a \emph{strong generator} if
    there exists an integer $ n \geq 0 $
    such that the equality $ \langle \Sigma^i C \mid
    i \in \ZZ \rangle^{\star n} = \cat{C} $ holds.
\end{definition}

Then we introduce the following ``big'' variant:

\begin{definition}\label{a669e3218b}
    Let $ \cat{C} $ be a stable $ \infty $-category admitting small coproducts.
    We call a collection of objects of~$ \cat{C} $
    \emph{big closed} if
    it is closed under
    small coproducts and direct summands.
    For a collection $ S \subset \cat{C} $,
    let $ \llangle S \rrangle $ denote
    the smallest big closed subcollection containing~$ S $.
\end{definition}

Note that if $ S $ and $ T $ are big closed,
$ S \star T $ is also big closed.

\begin{example}[G.\,M.\,Kelly]\label{kelly}
    For a commutative ring~$ R $ of global dimension~$ n $,
    we have
    $ \llangle \Sigma^i R \mid i \in \ZZ \rrangle^{\star (n+1)}
    = \QCoh(\Spec R) $
    by using arguments made in~\autocite{Kelly65}.
    See \autocite[Section~8]{Christensen98} for a detailed account.
\end{example}

The following result, which was proven in \autocite[Section~2]{Neeman},
explains why we care about the big variant.

\begin{theorem}[Neeman]\label{neeman}
    Let $ X $ be a noetherian scheme.
    Suppose that an object $ F \in \Coh(X) $
    satisfies
    $ \llangle \Sigma^i F
    \mid i \in \ZZ \rrangle^{\star n} = \QCoh(X) $ for some integer $ n \geq 0 $.
    Then $ F $ is a strong generator of $ \Coh(X) $.
\end{theorem}

\section{Boundedness}\label{s_bounded}

We introduce the notion of boundedness of a morphism between schemes.

\begin{definition}\label{a9fcbaa4d1}
    Let $ f \colon Y \to X $ be a morphism
    between noetherian schemes.
    It is called \emph{coherently bounded}
    if for every $ G \in \Coh(Y) $,
    there exists an object $ F \in \Coh(X) $
    and an integer $ n \geq 0 $ such that
    $ f_* G \in \llangle F \rrangle^{\star n} $ holds.
\end{definition}

In this section, we prove that
many morphisms are coherently bounded.

\begin{example}\label{671d42539c}
    Any proper morphism $ Y \to X $
    between noetherian schemes is coherently bounded
    since the direct image functor
    sends an object of $ \Coh(Y) $ into $ \Coh(X) $.
\end{example}

\begin{proposition}\label{632b179157}
    Any composition of
    two coherently bounded morphisms
    is coherently bounded.
\end{proposition}

\begin{proof}
    Suppose that $ f \colon Y \to X $ and
    $ g \colon Z \to Y $ are
    coherently bounded morphisms between noetherian schemes.
    For $ H \in \Coh(Z) $, we can take an object $ G \in \Coh(Y) $
    and an integer $ n \geq 0 $ satisfying $ g_*H \in \llangle G \rrangle^{\star n} $.
    Similarly, we can take an object $ F \in \Coh(Y) $
    and an integer $ m \geq 0 $ satisfying $ f_*G \in \llangle F \rrangle^{\star m} $.
    Then we have $ (f \circ g)_* H \simeq f_*(g_*H)
    \in f_*(\llangle G \rrangle^{\star n})
    \subset \llangle f_*G \rrangle^{\star n} \subset \llangle F \rrangle^{\star mn} $.
\end{proof}

\begin{lemma}\label{1a3f3b9d8c}
    Any open immersion between
    separated noetherian schemes is coherently bounded.
\end{lemma}

\begin{proof}
    Let $ j \colon U \to X $
    be an open immersion between separated noetherian schemes.
    Consider an object $ G \in \Coh(U) $.
    We wish to find an object $ F \in \Coh(X) $ and
    an integer~$ n \geq 0 $ satisfying
    $ j_*G \in \llangle F \rrangle^{\star n} $.
    Since $ G $ is a direct summand of
    some objects of the form $ j^*F' $ with $ F' \in \Coh(X) $,
    we may assume that $ G = j^*F' $ holds for some $ F' \in \Coh(X) $.
    According to \autocite[Theorem~6.2]{Neeman},
    there exists an object $ F'' \in \Perf(X) $
    and an integer $ n \geq 0 $ such that
    $ j_*\shv{O}_U \in \llangle F'' \rrangle^{\star n} $ holds.
    Then the pair consisting of $ F = F' \otimes F'' \in \Coh(X) $
    and this~$ n $ work since
    we have $ j_*G \simeq F' \otimes j_*\shv{O}_U
    \in F' \otimes \llangle F'' \rrangle^{\star n}
    \subset \llangle F' \otimes F'' \rrangle^{\star n}
    = \llangle F \rrangle^{\star n} $.
\end{proof}

\begin{theorem}\label{d310da1c37}
    Any morphism of finite type
    between separated noetherian schemes is coherently bounded.
\end{theorem}

\begin{proof}
    Nagata's compactification theorem says
    that such a morphism is factored into an open immersion
    followed by a proper morphism.
    So the desired result follows from
    \cref{671d42539c},
    \cref{632b179157}, and
    \cref{1a3f3b9d8c}.
\end{proof}

\section{Descendability}\label{s_descendable}

In this section,
the notion of descendability,
which is introduced
in \autocite[Section~3]{Mathew16},
and see how it is related to our problem.

We let $ \Cat{Pr}^{\textnormal{St}} $ denote the $ \infty $-category
whose objects are presentable $ \infty $-categories
and whose morphisms are colimit preserving functors.
We equip it with
the symmetric monoidal structure given in \autocite[Section~4.8.2]{HA}.

\begin{definition}[Mathew]\label{935da22422}
    Let $ \cat{C} $ be an object of $ \CAlg(\Cat{Pr}^{\textnormal{St}}) $;
    concretely, $ \cat{C} $ is a stable presentable $ \infty $-category
    equipped with a symmetric  monoidal structure whose
    tensor product operations preserve small colimits in each variable.
    A commutative algebra object $ A \in \CAlg(\cat{C}) $ is
    called \emph{descendable} if
    $ \cat{C} $ is the smallest thick tensor ideal containing~$ A $.
\end{definition}

\begin{example}[Bhatt--Scholze]\label{bs}
    Recall that a morphism $ f \colon Y \to X $
    between noetherian schemes is called
    an h~cover if it is of finite type and every base change
    is (topologically) submersive.
    According to \autocite[Proposition~11.25]{BhattScholze17},
    for such a morphism~$ f $,
    the direct image
    $ f_*\shv{O}_Y $ is descendable when viewed as
    a commutative algebra object of $ \QCoh(X) $.
\end{example}

The following characterization is standard;
see \autocite[Lemma~11.20]{BhattScholze17} for a proof.

\begin{proposition}\label{e552612dc4}
    For $ \cat{C} \in \CAlg(\Cat{Pr}^{\textnormal{St}}) $
    and $ A \in \CAlg(\cat{C}) $,
    let $ K $ denote the fiber of the canonical morphism
    $ \unit_{\cat{C}} \to A $.
    Then $ A $ is descendable
    if and only if $ K^{\otimes k} \to \unit_{\cat{C}} $
    is zero for some $ k \geq 0 $.
\end{proposition}

The following observation explains
why this notion is useful for us;
it says that a descendable commutative algebra object
generates the given stable $ \infty $-category
using a finite number of steps in some sense.

\begin{proposition}\label{2405c8441e}
    Consider $ \cat{C} \in \CAlg(\Cat{Pr}^{\textnormal{St}}) $
    and $ A \in \CAlg(\cat{C}) $ and suppose
    that $ A $ is descendable.
    Then $ \langle A \otimes C \mid C \in \cat{C} \rangle^{\star k}
    = \cat{C} $ holds for some integer $ k \geq 0 $.
\end{proposition}

\begin{proof}
    Let $ S $ denote the collection $ \langle A \otimes C
    \mid C \in \cat{C} \rangle $
    and $ K $ the fiber of the map $ \unit_{\cat{C}} \to A $.
    By \cref{e552612dc4}, there is an integer $ k \geq 0 $ such that
    the canonical morphism $ K^{\otimes k} \to \unit_{\cat{C}} $ is zero.
    Now we can check by induction on $ 0 \leq i \leq k $ that
    $ \cofib(K^{\otimes i} \to \unit_{\cat{C}}) \otimes C
    \in S^{\star i} $
    for every object $ C \in \cat{C} $;
    this is trivial when $ i = 0 $ and
    the inductive step follows from the cofiber sequence
    \begin{equation*}
        A \otimes
        (K^{\otimes i} \otimes C)
        \to \cofib(K^{\otimes (i+1)} \to \unit_{\cat{C}}) \otimes C
        \to \cofib(K^{\otimes i} \to \unit_{\cat{C}}) \otimes C.
    \end{equation*}
    Now we have $ S^{\star k} = \cat{C} $
    from the case when $ i = k $
    since $ \unit_{\cat{C}} $ is a direct summand of
    $ \cofib(K^{\otimes k} \xrightarrow{0} \unit_{\cat{C}})
    \simeq \Sigma K^{\otimes k} \oplus \unit_{\cat{C}} $.
\end{proof}

\begin{corollary}\label{060b8c5d40}
    Let $ L \colon \cat{C} \to \cat{D} $
    be a morphism in $ \CAlg(\Cat{Pr}^{\textnormal{St}}) $.
    Assume that the right adjoint~$ R $ of~$ L $
    preserves small colimits and
    the pair $ (L, R) $
    satisfies the projection formula;
    that is, the canonical morphism
    \begin{equation*}
        C \otimes R(D) \to R(L(C) \otimes D)
    \end{equation*}
    is an equivalence for any $ C \in \cat{C} $ and $ D \in \cat{D} $.

    Suppose that $ S \subset \cat{D} $ is a subcollection
    satisfying $ \llangle S \rrangle^{\star n} = \cat{D} $ for some integer $ n \geq 0 $.
    Then
    if $ R(\unit_{\cat{D}}) $
    is descendable, $ \llangle R(S) \rrangle^{\star m} = \cat{C} $
    holds for some integer $ m \geq 0 $,
    where $ R(S) $ denotes the (set theoretic) direct image of~$ S $ under~$ R $.
\end{corollary}

\begin{proof}
    By the projection formula,
    we see that
    $ R(\cat{D}) $ contains $ R(\unit_{\cat{D}}) \otimes C \simeq R(L(C)) $
    for every object $ C \in \cat{C} $.
    Hence by \cref{2405c8441e}, there is an integer $ k \geq 0 $
    satisfying $ \llangle R(\cat{D})\rrangle^{\star k} = \cat{C} $.
    Combining this with
    $ R(\cat{D}) = R(\llangle S \rrangle^{\star n})
    \subset \llangle R(S) \rrangle^{\star n} $,
    we have $ \llangle R(S) \rrangle^{\star nk} = \cat{C} $.
\end{proof}

\section{Proof of Main Theorem}\label{s_main}

By using the observations made so far,
we obtain the following:

\begin{theorem}\label{main}
    Let $ f \colon Y \to X $ be an h~cover
    between quasicompact separated noetherian schemes.
    Suppose that there is an
    object~$ G \in \Coh(Y) $
    satisfying $ \llangle \Sigma^i G \mid i \in \ZZ
    \rrangle^{\star n} = \QCoh(Y) $ for some integer $ n \geq 0 $.
    Then
    there exists an object $ F \in \Coh(X) $
    and an integer $ m \geq 0 $ such that $ \llangle \Sigma^i F
    \mid i \in \ZZ \rrangle^{\star m} = \QCoh(X) $ holds.
\end{theorem}

\begin{proof}
    By \cref{d310da1c37},
    we can take an object $ F \in \Coh(X) $
    and an integer $ k \geq 0 $ such that
    $ f_*G \in \llangle F \rrangle^{\star k} $ holds.
    On the other hand, by \cref{bs}, we can apply \cref{060b8c5d40}
    to see that there is an integer $ l \geq 0 $
    satisfying $ \llangle \Sigma^i f_*G \mid i \in \ZZ \rrangle^{\star l}
    = \QCoh(X) $.
    From these observations,
    we have $\llangle \Sigma^i F \mid i \in \ZZ \rrangle^{\star kl}
    = \QCoh(X) $.
\end{proof}

We conclude this paper by proving
Main Theorem, which we stated in \cref{s_intro}.

\begin{proof}[Proof of Main Theorem]
    Gabber's weak local uniformization theorem says
    that there exists an h~cover $ Y \to X $
    where $ Y $ is the spectrum of a regular ring,
    which automatically has finite (global) dimension;
    see~\autocite{IllusieLaszloOrgogozo14} for a proof.
    Hence from \cref{kelly} we can apply \cref{main}
    to see that there exists an object $ F \in \Coh(X) $
    satisfying $ \llangle \Sigma^i F \mid i \in \ZZ \rrangle^{\star n} = \Coh(X) $
    for some integer~$ n \geq 0 $. According to \cref{neeman},
    this implies the desired result.
\end{proof}

\bibliographystyle{spmpsci}
\bibliography{references.bib}

\end{document}